\documentclass[twoside,bezier,11pt, reqno]{amsart}
\usepackage{amsmath, amsthm, amscd, amsfonts,url, amssymb,graphicx,graphics, color}
\usepackage[bookmarksnumbered, plainpages]{hyperref}
\theoremstyle{plain}
\newtheorem{theorem}{Theorem}[section]
\newtheorem*{theorem*}{Theorem}

\newtheorem{remark}{Remark}
\newtheorem{prop}[theorem]{Proposition}

\newtheorem*{mt*}{Main Theorem}
\newcommand\C{{\mathbb C}}

\newcommand\g{{\mathfrak{g}}}
\newcommand\h{{\mathfrak{h}}}

\newcommand\su{{\mathfrak{su}}}

\newcommand\m{{\mathfrak{m}}}

\renewcommand\t{{\theta}}

\newcommand\s{{\mathbb S}}

\newcommand{\D}{\mathcal D}
\newcommand{\A}{\mathcal A}

\numberwithin{equation}{section}
\begin{document}
\begin{center}
{\bf\large Some geometrical properties of Berger Spheres}
 \\[0.5cm]
{Y. AryaNejad \footnote{ Corresponding author: y.aryanejad@pnu.ac.ir\\ \qquad Received: xx Month 201x\\ \qquad  Revised: xx Month 201x \\ \qquad Accepted: xx Month 201x}\\ 
Department of Mathematics, Payame noor University, P.O. Box 19395-3697, Tehran, Iran.} \\[2mm]

\end{center}%
\vspace*{0.5cm}
%
\begin{quotation}
\noindent
{\footnotesize {\sc Abstract.}
Our aim in this paper is to investigate some geometrical properties of Berger Spheres i.e. homogeneous
Ricci solitons and harmonicity properties of invariant vector fields. We determine all vector fields which are critical points for the energy functional
restricted to vector fields. We also see that do not exist any vector fields defining harmonic map, and the energy of critical points is explicitly calculated.\\


}
\end{quotation}
\markboth {Y. AryaNejad}{Some geometrical properties of Berger Spheres}
\section{ Introduction}
\noindent 
Suppose $\lbrace P,X_1 ... X_m,Y_1 ... Y_m,Q \rbrace$ is a basis for Lie
In Riemannian geometry, a Berger sphere, is a standard 3-sphere with Riemannian metric from a one-parameter family, which can be obtained from the standard metric by shrinking along fibers of a Hopf fibration. 
 These spaces, found by M. Berger \cite{B1}  in his classification of all simply
connected normal homogeneous Riemannian manifolds of positive sectional curvature,
have not constant curvature, and their metrics are obtained from the round metric on
$S^3$ by deforming it along the fibers of the Hopf fibration $S^3\rightarrow S^2$ by $\epsilon$.
In \cite{g1} the homogeneous Riemannian structures on the 3-dimensional Berger spheres,
their corresponding reductive decompositions and the associated groups of isometries
have been obtained.
These spaces are of great interest in Riemannian geometry and provide nice examples; for instance,
they served as counterexamples to a conjecture of Klingenberg about closed geodesics \cite{p11} and to conjectures on the first eigenvalue of the Laplacian on spheres
(\cite{b1}, \cite{u1}).
The Berger spheres $S_{\epsilon}^3$
are homogeneous Riemannian spaces diffeomorphic to the 3-
dimensional sphere.
The metrics so constructed are known as Berger
metrics, they consist in a 1-parameter variation $g_{\epsilon}$ for $\epsilon >0$. We
will consider also the Lorentzian Berger metrics, i.e. when $\epsilon <0$.
 Up to our knowledge, no geometrical properties such as homogeneous
Ricci solitons have been obtained yet for Berger Spheres.\\
 A natural generalization of an Einstein manifold is Ricci soliton, i.e.   
a pseudo Riemannian metric $g$ on a smooth manifold $M$, such that the equation
\begin{eqnarray}\label{ric}
\begin{array}{cccc}
\mathcal{L}_{X} g =\varsigma g-\varrho,
\end{array}
\end{eqnarray}
holds for some $\varsigma \in R$ and some smooth vector field $X$ on $M$, where $\varrho$
denotes the Ricci tensor of $(M, g)$ and $\mathcal{L}_{X}$ is the usual Lie derivative.
 According
to whether $\varsigma > 0, \varsigma = 0$ or $\varsigma < 0$, a Ricci soliton $g$ is said to be shrinking, steady or expanding, respectively. A homogeneous Ricci soliton on a homogeneous space $M = G/H$ is a G-invariant metric $g$ for which the equation \eqref{ric} holds and
an invariant Ricci soliton is a homogeneous apace, such that equation \eqref{ric} is satisfied by an invariant vector field.
Indeed, the study of Ricci solitons homogeneous spaces is an interesting area
of research in pseudo-Riemannian geometry.
For example, evolution of homogeneous Ricci solitons under the bracket flow \cite{Lm}, algebraic solitons and the Alekseevskii Conjecture properties\cite{LM}, conformally flat Lorentzian gradient Ricci
solitons\cite{MB}, properties of algebraic Ricci Solitons of three-dimensional Lorentzian
Lie groups\cite{BA}, algebraic Ricci solitons \cite{Ba}. Non-K\"{a}hler examples of Ricci solitons are very hard to find yet (see \cite{DH}).
In case $(G, g)$ be a simply-connected completely solvable Lie
group equipped with a left-invariant metric, and $(\g,\langle,\rangle )$ be the corresponding
metric Lie algebra, then $(G, g)$ is a Ricci soliton if and only if $(\g,\langle ,\rangle)$ is an
algebraic Ricci soliton \cite{LJ}. \\
On the other hand, investigating critical points of the energy associated to vector fields is an interesting purpose
under different points of view. As an example by the
Reeb vector field $\xi$ of a contact metric manifold, somebody can see how the criticality of such a vector field is related to the geometry of the manifold (\cite{p1},\cite{p2}).
Recently, it has been \cite{g3} proved that critical points of $E:\mathfrak{X}(M)\rightarrow R$, that is,
the energy functional restricted to vector fields, are again parallel vector fields. Moreover,
in the same paper it also has been determined the tension field associated to a unit vector field $V$, and investigated the problem of determining when $V$ defines a harmonic map. \\
A Riemannian manifold admitting a parallel vector field is locally reducible, and the same is true for a
pseudo-Riemannian manifold admitting an either space-like or time-like parallel vector field. This leads us to consider different situations, where some interesting types of non-parallel vector fields can be characterized in terms of harmonicity properties. We may refer
to the recent monograph \cite{d11} and some references \cite{i}, \cite{n} for an overview on harmonic vector fields.\\
This paper is organized in the following way. We devote section 2 to recall the definitions and to state the results we will need in the sequel.
In Section 3, we investigate rquired conditions for Berger Spheres Ricci solitons. Harmonicity properties
of vector fields on Berger Spheres will be determined
in Sections 4. Finally, the energy of all these vector fields is explicitly calculated in Section 5.
\section{preliminaries}
Let $M = G/H$ be a homogeneous manifold (with $H$ connected), $\g$ the Lie algebra
 of $G$ and $\h$ the isotropy subalgebra. Consider $\m =\g /\h$ the factor space, which identies
with a subspace of $\g$ complementary to $\h$. The pair $(\g,\h)$ uniquely defines the isotropy
representation
\begin{center}
$\psi :\g \longrightarrow \mathfrak{gl}(\m),\quad      \psi(x)(y)=[x,y]_\m$   
\end{center}
for all  $x\in \g, y\in \m$. Suppose that $\lbrace e_1,...,e_r,u_1,...,u_n\rbrace$ be a basis of $\g$, where $\lbrace e_j\rbrace$ and $\lbrace u_i\rbrace$ are bases of $\h$ and $\m$
respectively, then with respect to $\lbrace u_i \rbrace$, $H_j$ whould be the isotropy representation  for $e_j$.
$g$ on $\m$ uniquely defines its invariant linear Levi-Civita connection, as the corresponding homomorphism of $\h$-modules  $\Lambda:\g \longrightarrow \mathfrak{gl}(\m)$ such that $\Lambda(x)(y_\m)=[x,y]_\m$ for all  $x\in \h, y\in \g$. In other word
\begin{eqnarray}\label{con}
\begin{array}{cccc}
\Lambda(x)(y_\m)=\frac{1}{2}[x,y]_\m+v(x,y)
\end{array}
\end{eqnarray}
for all  $x,y\in \g$, where $v:\g \times \g\rightarrow \m$ is the $\h$-invariant symmetric mapping uniquely determined by
\begin{center}
$2g(v(x, y), z_\m) = g(x_\m, [z,y]_\m) + g(y_\m,[z,x]_\m)$
\end{center}
for all  $x,y,z\in \g$,
Then the curvature tensor can be determined by
\begin{eqnarray}\label{cur}
R:\m \times \m \longrightarrow \mathfrak{gl}(\m),\quad & R(x,y)=[\Lambda(x),\Lambda(y)]-\Lambda ([x,y]),   
\end{eqnarray}
and with respect to $u_i$, the Ricci tensor $\rho$ of $g$  is given by
\begin{eqnarray}\label{ric2}
\rho (u_i ,u_j)=\sum_{k=1}^4g(R(u_k,u_i)u_j,u_k),\quad i,j=1,\dots,4.
\end{eqnarray}
Let $(M,g)$ be a compact Riemannian manifold and $g_s$ be the Sasaki metric
on the tangent bundle $TM$, then the energy of a smooth vector field $V:(M,g)\longrightarrow
(TM,g^s)$ on $M$ is;
\begin{equation}\label{enr}
E(V)=\dfrac{n}{2}vol (M,g)+\dfrac{1}{2}\int_M ||\nabla V||^2dv
\end{equation}
(assuming M compact; in the non-compact case, one works over relatively compact domains see \cite{c1}). If $V:(M,g)\longrightarrow (TM,g^s)$ be a critical point for the energy functional, then $V$ is said to define a harmonic map. The Euler-Lagrange equations characterize vector fields $V$ defining harmonic maps as the ones whose tension field $\t(V)=tr(\nabla^2V)$ vanishes.
 Consequently,
$V$ defines a harmonic map from $(M,g)$ to $(TM,g^s)$ if and only if
\begin{equation}\label{hor}
 tr[R(\nabla_. V,V)_.]=0, \quad  \nabla^*\nabla V=0,
\end{equation}
where with respect to an orthonormal local frame $\lbrace e_1,...,e_n\rbrace$ on $(M,g)$, with $\varepsilon_i=g(e_i,e_i)=\pm1$ for all indices i,
one has
\begin{center}
$ \nabla^*\nabla V=\sum_i \varepsilon_i( \nabla_{e_i}\nabla_{e_i} V-\nabla_{\nabla_{e_i}e_i}V)$.   
\end{center}
A smooth vector field V is said to be a harmonic
section if and only if it is a critical point of $E^v(V)=(1/2)\int_M||\nabla V||^2dv$ where $E^v$ is the vertical energy. The corresponding Euler-Lagrange equations are given by
\begin{equation}
 \nabla^*\nabla V=0,
\end{equation}
Let $\mathfrak{X}^{\rho}(M) =\lbrace V\in \mathfrak{X}(M): ||V||^2=\rho^2 \rbrace$ and $\rho\neq 0$. Then, one can consider vector fields $ V\in \mathfrak{X}(M)$
which are critical points for the energy functional $E
|_{\mathfrak{X}^{\rho}(M)}$, restricted to vector fields of the same constant length. The
Euler-Lagrange equations of this variational condition are given by
\begin{equation}\label{hor1}
\nabla^*\nabla V\quad is\quad collinear\quad to\quad V.   
\end{equation}
As usual, condition \eqref{hor1} is taken as a definition
of critical points for the energy functional restricted to vector fields of the same length in the non-compact case.

\section{Homogeneous Ricci solitons on Berger spheres}
Following \cite{g1} the sphere $\s^3$ and the Lie group $SU(2)$ have been identified by the map that sends
$(z,w) \in \s^3 \subset \C^2$ to
$$
 \left( \begin{array}{cc}
   z & w   \\
   -\bar{w} & \bar{z}   \\
   \end{array}  \right)  \in SU(2).
 $$
We consider the basis $\lbrace X_1,X_2,X_3 \rbrace$ of the Lie
algebra $\su(2)$ of $SU(2)$ given by 
\begin{equation}
 X_1=\left( \begin{array}{cc}
   i & 0   \\
   0 & \bar{i}   \\
   \end{array}  \right), \quad
   X_2=\left( \begin{array}{cc}
   0 & 1   \\
   -1 & 0   \\
   \end{array}  \right) \quad
   X_3=\left( \begin{array}{cc}
   0 & i   \\
   i & 0   \\
   \end{array}  \right).
 \end{equation}
Then, the Lie brackets are determined by
\begin{equation}
 [X_1,X_2]=2X_3,\quad [X_2,X_3]=2X_1,\quad  [X_3,X_1]=2X_2.
 \end{equation}
The one-parameter family $\lbrace g_{\epsilon} : \epsilon>0\rbrace$ of left-invariant Riemannian metrics on $\s^3 =
SU(2)$ given at the identity, with respect to the basis of left-invariant vector fields
$ X_1,X_2,X_3 $ by
\begin{equation}\label{meter}
g_{\epsilon}= \left( \begin{array}{ccc}
   \epsilon & 0 & 0   \\
   0 & 1 & 0   \\
   0 & 0 & 1  
   \end{array}  \right),
\end{equation}
are called the Berger metrics on $\s^3$; if $\epsilon=1$ we have the canonical (bi-invariant) metric and for $\epsilon<0$ the one-parameter family $\lbrace g_{\epsilon} : \epsilon<0\rbrace$ are left-invariant Lorentzian metrics.
The Berger spheres are the simply connected complete Riemannian manifolds $\s_{\epsilon}^3=(\s^3,g_{\epsilon})$, $\epsilon>0$. We will use the name "Lorentzian Berger spheres" for case $\epsilon<0$. \\
Setting $\Lambda_i=\nabla_{e_i}$, the components of the Levi-Civita connection are calculated using the well known {\em Koszul} formula and are
\begin{equation}
 \Lambda_1= \left( \begin{array}{ccc}
   0 & 0 & 0   \\
   0 & 0 & \epsilon -2   \\
   0 & 2-\epsilon & 1  
   \end{array}  \right),
   \end{equation}
  $$
   \Lambda_2= \left( \begin{array}{ccc}
   0 & 0 & 1   \\
   0 & 0 & 0   \\
   -\epsilon & 0 & 0  
   \end{array}  \right), \quad \quad 
   \Lambda_3= \left( \begin{array}{ccc}
   0 & -1 & 0   \\
   \epsilon & 0 & 0   \\
   0 & 0 & 0 
   \end{array}  \right).
$$
Using \eqref{cur} we can determine the non-zero curvature components;
$$
 \begin{array}{ccc}
   R(X_1,X_2)X_1=\epsilon^2 X_2 ,& R(X_1,X_2)X_2=-\epsilon X_1,& R(X_1,X_3)X_3=-\epsilon X_1,  \\
   R(X_1,X_3)X_1=\epsilon^2 X_3, & R(X_2,X_3)X_2=(4-3\epsilon)X_3 ,& R(X_2,X_3)X_3=(3\epsilon-4)X_2.
     \end{array} 
    $$
Since $R(X, Y,Z,W) = g(R(X, Y )Z,W)$ we have;
$$
 \begin{array}{cc}
   R(X_1,X_2,X_1, X_2)= R(X_1,X_3,X_1, X_3)=\epsilon^2 ,& R(X_2,X_3,X_2,X_3)=(4-3\epsilon).
     \end{array} 
    $$
Applying the Ricci tensor formula \eqref{ric2}, we get;
\begin{equation}\label{ric3}
(\rho)_{ij}= \left( \begin{array}{ccc}
   2\epsilon^2 & 0 & 0   \\
   0 & 4-2\epsilon & 0   \\
   0 & 0 & 4-2\epsilon  
   \end{array}  \right),
\end{equation}
which is diagonal with eigenvalues $r_1=2\epsilon^2$ and $r_2=r_3=  4-2\epsilon$.\\
For an arbitrary left-invariant vector field $X =aX_1+bX_2+cX_3\in \su(2)$ we have;
$$
 \begin{array}{ccc}
  \nabla_{X_1}X= (\epsilon-2)(cX_2-bX_3) ,& \nabla_{X_2}X=cX_1-a\epsilon X_3 ,& \nabla_{X_3}X=-bX_1+a\epsilon X_2.
     \end{array}
$$
Using the relation $(L_{X}g)(Y,Z) = g(\nabla_Y X,Z) + g(Y,\nabla_ZX)$ we have;
 \begin{equation}\label{kil}
  L_{X}g= \left( \begin{array}{ccc}
   0 & 2(1-\epsilon)c & 2(\epsilon-1)b   \\
   2(1-\epsilon)c & 0 & 0   \\
   2(\epsilon-1)b & 0 & 0  
   \end{array}  \right).
\end{equation}
By the Ricci soliton formula \eqref{ric}, we get the following
system of differential equations;
  \begin{equation}\label{eric}
 \begin{array}{cr} 
    \hspace*{0mm} 2(1-\epsilon)c=0,\\
  \hspace*{0mm} 2(\epsilon-1)b=0,\\
  \hspace*{0mm} 2\epsilon^2-\lambda \epsilon=0,\\
  \hspace*{0mm} 4-2\epsilon-\lambda =0
   \end{array}
\end{equation} 
If $\epsilon=1$, then from \eqref{ric3} we can see $\rho_{ij}=\lambda g_{ij}$ for all indices i, j and therefore $g_1$ is an Einstein metric on $\s^3$. So, we suppose that $\epsilon\neq1$. From the first and the second equations in \eqref{eric} we get $b=c=0$ and the third and
the last equations in \eqref{eric} give us $\epsilon=0$ and $\lambda=4$. Therefore the only solution occurs when $\epsilon=0$, and for an arbitrary $\epsilon\neq 0,1$ the system of differential equations \eqref{eric} is incompatible. Thus, we have the following.
\begin{prop}\label{pro1}
Let $(\s^3, g_{\epsilon})$ be a Berger sphere $(\epsilon >0)$. Then $(\s^3, g_{\epsilon})$ can not be a homogeneous Ricci soliton manifold.
\end{prop}
We proved the following result too.
\begin{prop}\label{pro2}
Let $(\s^3, g_{\epsilon})$ be a Lorentzian Berger sphere $(\epsilon<0)$. Then $(\s^3, g_{\epsilon})$ can not be a homogeneous Ricci soliton manifold.
\end{prop}
\begin{remark}
 Proposition \eqref{pro2} confirms the classification result in \cite{MB}, while Proposition \eqref{pro1} emphasizes that there are no left-invariant Ricci solitons on three-dimensional Riemannian Lie groups \cite{d111}.
 \end{remark}
A D' Atri space is defined as a Riemannian manifold $(M, g)$ whose local geodesic symmetries are volumepreserving. Let us recall that the property of being a D' Atri space is
equivalent to the infinite number of curvature identities called the odd Ledger conditions $L_{2k+1}$, $k\geq 1$ (see \cite{d1}, \cite{s1}). In particular, the two first non-trivial Ledger conditions are:
\begin{equation}
L_3: (\nabla_X \rho)(X,X)=0,\quad L_5: \sum_{a,b=1}^nR(X,E_a,X,E_b)(\nabla_X R)(X,E_a,X,E_b)=0,
\end{equation}
where $X$ is any tangent vector at any point $m\in M$ and $\lbrace E_1, . . . , E_n\rbrace$ is any orthonormal basis of $T_mM$. Here $R$ denotes
the curvature tensor and $\rho$ the Ricci tensor of $(M, g)$, respectively, and $n = dimM$.\\
Thus, it is natural to start with the investigation of all homogeneous Riemannian Berger spheres satisfying the simplest Ledger condition $L_3$, which is the first approximation of the D' Atri property. This condition is called in \cite{p112} "the class
$\A$ condition". Equivalently Ledger condition $L_3$ holds if and only if
the Ricci tensor is cyclic-parallel, i.e. $ (\nabla_X \rho)(Y,Z)+ (\nabla_Y \rho)(Z,X)+ (\nabla_Z \rho)(X,Y)=0$.
For more detail see \cite{c1}.
\begin{prop} 
Let $(\s^3, g_{\epsilon})$ be a Berger sphere. Then $(\s^3, g_{\epsilon})$ is a D' Atri space which its first approximation holds.
\end{prop}
\begin{proof}
In Ledger condition $L_3$,
\begin{center}
$\nabla_i\rho_{jk}=-\sum_t(\varepsilon_jB_{ijt}\rho_{tk}+\varepsilon_k B_{ikt}\rho_{tj}),$
\end{center}
where  $B_{ijk}$ components can be obtained by the relation $\nabla_{e_i}e_j =\sum_k\varepsilon_j B_{ijk}e_k$ with $\varepsilon_i=g(e_i,e_i)=\pm1$ for all indices i.
Hence $\nabla_1\rho_{11}=\nabla_2\rho_{22}=\nabla_3\rho_{33}=0$, as desired.
\end{proof}
A pseudo-Riemannian manifold which admits a parallel degenerate distribution is called a {\em Walker} manifold. Walker spaces were introduced by Arthur Geoffrey Walker in 1949. The existence of such structures causes many interesting properties for the manifold with no Riemannian counterpart. Walker also determined a standard local coordinates for these kind of manifolds \cite{Wa1,Wa2}.
\begin{prop} 
Let $(\s^3, g_{\epsilon})$ be a Lorentzian Berger sphere $(\epsilon<0)$. Then $(\s^3, g_{\epsilon})$ can not be a Walker manifold..
\end{prop}
\begin{proof}
Set $X =aX_1+bX_2+cX_3\in \su$ and suppose that $\D={\rm span}(X)$ is an invariant null parallel line field. Then, the following equations must satisfy for some parameters $\omega_1,\dots,\omega_3$
$$
\begin{array}{lll}
\nabla_{X_1}X=\omega_1X,&\nabla_{X_2}X=\omega_2X,&\nabla_{X_3}X=\omega_3X.
\end{array}
$$
By straight forward calculations we conclude that the following equations must satisfy
$$
\begin{array}{ccc}
\omega_1a=0,&\quad \omega_1b+c(2-\epsilon)=0,&\quad\omega_1c+b(\epsilon-2)=0,\\
\omega_2b=0,&\quad \omega_2a-c=0,&\quad
 -\omega_2c+a\epsilon=0,\\
\omega_3c=0,&\quad \omega_3a+b=0,&\quad
\omega_3b-a=0.
\end{array}
$$
$X$ is null, hence $X$ must satisfy $g(X,X)=a^2\epsilon+b^2+c^2=0$ described in \eqref{meter}. For solving the above system of equations as we can see, since $\epsilon\neq 0$, a non-trivial solution can not occur.
  \end{proof}
  
\section{Harmonicity of vector fields on Berger spheres}
In this section we investigate the harmonicity of invariant vector fields on a Berger sphere  $(\s^3, g_{\epsilon})$, in both Riemannian $(\epsilon>0)$ and Lorentzian $(\epsilon<0)$.
We treat separately the case when g is Lorentzian
and the Riemannian case.
\subsection{Lorentzian case}
In Lorentzian case $(\epsilon<0)$ we can construct an orthonormal frame field $\lbrace e_1,e_2,e_3\rbrace$ with respect to $g_\epsilon$;
\begin{equation}\label{base}
e_1=\frac{1}{\sqrt{-\epsilon}}X_1,\quad e_2=X_2,\quad e_3=X_3.
\end{equation}
and we get;
\begin{equation}
 [e_1,e_2]=\frac{2}{\sqrt{-\epsilon}}e_3,\quad [e_2,e_3]=2\sqrt{-\epsilon} e_1,\quad  [e_3,e_1]=\frac{2}{\sqrt{-\epsilon}}e_2.
 \end{equation}
Considering formula \eqref{con} the connection components are;
\begin{eqnarray}\label{con1}
\begin{array}{cr}
\nabla_{e_1}e_2=\frac{2-\epsilon}{\sqrt{-\epsilon}}e_3,\quad \nabla_{e_1}e_3=\frac{\epsilon-2}{\sqrt{-\epsilon}}e_2,\quad \nabla_{e_2}e_1=\sqrt{-\epsilon}e_3,\\
\nabla_{e_2}e_3=\sqrt{-\epsilon}e_1,\quad \nabla_{e_3}e_1=-\sqrt{-\epsilon}e_2,\quad \nabla_{e_3}e_2=-\sqrt{-\epsilon}e_1,
\end{array}
\end{eqnarray}
while $\nabla_{e_i}e_j=0$ in the remaining cases.\\
For an arbitrary left-invariant vector field $V =ae_1+be_2+ce_3 \in \su(2)$ we can now use \eqref{con1}  to calculate $\nabla_{e_i}V$ for all indices i. We get
\begin{eqnarray}\label{con2}
\begin{array}{cr}
\nabla_{e_1}V=\frac{\epsilon-2}{\sqrt{-\epsilon}}ce_2+\frac{2-\epsilon}{\sqrt{-\epsilon}}be_3,\quad\nabla_{e_2}V=\sqrt{-\epsilon}ce_1+\sqrt{-\epsilon}ae_3,\\
\hspace*{0mm}\nabla_{e_3}V=-\sqrt{-\epsilon}be_1-\sqrt{-\epsilon}ae_2.
\end{array}
\end{eqnarray}
From \eqref{con2} it follows at once that there are no parallel vector fields $V\neq 0$ belonging to $\su(2)$.
We can now calculate $\nabla_{e_i}\nabla_{e_i}V$ and $\nabla_{\nabla_{e_i}e_i}V$ for all indices i. We obtain
\begin{eqnarray}
\begin{array}{cc}\label{con3}
\nabla_{e_1}\nabla_{e_1}V=\frac{(\epsilon-2)^2}{\epsilon}(be_2+ce_3),& \nabla_{e_2}\nabla_{e_2}V=-\epsilon(ae_1+ce_3),\\
\nabla_{e_3}\nabla_{e_3}-\epsilon(ae_1+be_2)),&
\nabla_{\nabla_{e_1}e_1}V=0,\\
\nabla_{\nabla_{e_2}e_2}V=0,&\quad\nabla_{\nabla_{e_3}e_3}V=0.
\end{array}
\end{eqnarray}
Thus, we find
\begin{equation}\label{nab1}
 \nabla^*\nabla V=\sum_i \varepsilon_i( \nabla_{e_i}\nabla_{e_i} V-\nabla_{\nabla_{e_i}e_i}V)
=-2\epsilon ae_1-2(\frac{\epsilon^2-2\epsilon+2}{\epsilon})(be_2+ce_3).   
\end{equation}
Since $\nabla^*\nabla V=-2(\frac{\epsilon^2-2\epsilon+2}{\epsilon})V+\frac{4-4\epsilon}{\epsilon}ae_1$, condition \eqref{hor1} results that $a=0$.
In the other direction, let $V=-2(\frac{\epsilon^2-2\epsilon+2}{\epsilon})(be_2+ce_3)$.  A direct calculation
yields that $\nabla^*\nabla V=-2(\frac{\epsilon^2-2\epsilon+2}{\epsilon})V$.\\
On the other hand, from \eqref{nab1} since $\nabla^*\nabla V=-2\epsilon V+4\frac{\epsilon-1}{\epsilon}(be_2+ce_3)$, then \eqref{hor1} results that $b=c=0$. 
Vice versa, let $V=ae_1$. By standard calculations we obtain $\nabla^*\nabla V=-2\epsilon V$. Thus, we have the following.
\begin{theorem}\label{hor4}
Let $(\s^3, g_{\epsilon})$ be a Lorentzian Berger sphere and  $V=ae_1+be_2+ce_3\in \su(2)$ be a left-invariant vector field on the Berger sphere for some real constants $a,b,c$.
\begin{itemize}
\item[$(a)$] 
$V$ is a critical point for the energy functional restricted to vector fields of the same length if and only if $V=-2(\frac{\epsilon^2-2\epsilon+2}{\epsilon})(be_2+ce_3)$. However, none of these vector fields is harmonic (in particular, defines a harmonic map).
\item[$(b)$] $V$ is a critical point for the energy functional restricted to vector fields of the same length if and only if $V=-2\epsilon ae_1$. However, none of these vector fields is harmonic (in particular, defines a harmonic map).
\end{itemize}
\end{theorem}
\subsection{Riemannian case}
Consider now a Riemannian Berger sphere $(\s^3, g_{\epsilon})$, $\epsilon>0$. $\su(2)$ admits a pseudo-orthonormal frame field $\lbrace e_1,e_2,e_3\rbrace$, where
\begin{equation}\label{base3}
e_1=\frac{1}{\sqrt{\epsilon}}X_1,\quad e_2=X_2,\quad e_3=X_3.
\end{equation}
The calculations are rather similar to the Lorentzian case. For this reason, details will be omitted. For an
arbitrary vector field $V=ae_1+be_2+ce_3\in \su(2)$, we find
\begin{equation}\label{base2}
 \nabla^*\nabla V=-2\epsilon ae_1-2(\frac{\epsilon^2-2\epsilon+2}{\epsilon})(be_2+ce_3),  
\end{equation}
which is similar to the Lorentzian case. So, we proved the
following.
\begin{theorem}\label{hor55}
Each vector field $V=ae_1$ and $V=be_2+ce_3\in \su(2)$ on a Riemannian Berger sphere $(\s^3, g_{\epsilon})$, $\epsilon>0$ are a critical point
for the energy functional restricted to vector fields of the same length. However, none of these vector fields is harmonic (in particular, defines a harmonic map).
\end{theorem}
Therefore by theorems \ref{hor4} and \ref{hor55} there is no harmonic vector field nor harmonic map on Berger spheres.
\subsection{Geodesic and Killing vector fields}
A vector field $V$ is geodesic if $\nabla_VV =0$, and is Killing if $\mathcal{L}_V g=0$, where $\mathcal{L}$ denotes
the Lie derivative i.e. $X$ is Killing if and only if  $g(\nabla_Y X,Z) + g(Y,\nabla_ZX)=0$ for all $Y,Z \in
\mathfrak{X}(M)$(the equation above is called the Killing equation).
 Parallel vector fields are both geodesic and Killing, and vector fields with these special geometric features often have particular harmonicity properties \cite{a2,g11,g2,h1}.
By standard calculations we obtain the following result.
\begin{prop} \label{pro1} 
Let $(\s^3, g_{\epsilon})$ be the Berger sphere and  $V\in \su(2)$ be a left-invariant vector field on the Berger sphere, then $V$ is geodesic if and only if $V=be_2+ce_3$ or $V=ae_1$. Moreover, by \eqref{kil} $V$ is Killing if and only if $V=ae_1$.
\end{prop}
In particular, from theorems \ref{hor4}, \ref{hor55} and proposition \ref{pro1} 
a straightforward calculation proves the following main classification result.
\begin{theorem}\label{hor6}
For a vector field $V=ae_1+be_2+ce_3\in \su$ on the $(\s^3, g_{\epsilon})$, the
following conditions are equivalent:
 \begin{itemize}
\item[$(1)$] $V$ is geodesic;
\item[$(2)$] $V$ is a critical point for the energy functional restricted to vector fields of the same length;
\item[$(3)$] $V=be_2+ce_3$ or $V=ae_1$.
\item[$(4)$] $V$ is Killing if and only if $b=c=0$.
\end{itemize}
\end{theorem}

\section{The energy of vector fields on Berger spheres}
We calculate explicitly the energy of a vector field $V\in\su(2)$ of a Berger sphere. This
gives us the opportunity to determine some critical values of the energy functional on Berger spheres. We shall first discuss geometric properties of the map $V$ defined by a
vector field  $V\in\su(2)$.
\begin{prop}\label{E1}
Let $(\s^3, g_{\epsilon})$ be a Berger sphere, $V=ae_1+be_2+ce_3\in \su(2)$ be a left-invariant vector field on the Berger sphere for some real constants $a,b,c$. Denote by $E(V)$ the energy of $V$. 
 \begin{itemize}
\item[$(a)$]  In the Lorentzian case $(\epsilon<0)$ the energy of $V$ is
\begin{center}
 $E(V)=(2+\frac{(\epsilon-2)^2+\epsilon^2}{2\epsilon}||V||^2-2\frac{\epsilon-1}{\epsilon} a^2)vol(\s^3, g_{\epsilon})$.  
\end{center}
\item[$(b)$] In the Riemannian case $(\epsilon>0)$ the energy of $V$ is
\begin{center}
 $E(V)=(2+\frac{(\epsilon-2)^2+\epsilon^2}{2\epsilon}||V||^2+2\frac{\epsilon-1}{\epsilon} a^2)vol(\s^3, g_{\epsilon})$.  
\end{center}
\end{itemize}
\end{prop}
\begin{proof}
Let $(\s^3, g_{\epsilon})$ be the Berger sphere. For case (a) consider the local orthonormal
basis $\lbrace e_1,e_2,e_3\rbrace$ of vector fields described in \eqref{base}. Then, locally,
\begin{center}
$||\nabla V||^2=\sum_{i=1}^{n}\varepsilon_i g(\nabla_{e_i}V,\nabla_{e_i}V).$
\end{center}
Let $V \in \su(2)$ be a left-invariant vector field on the Berger sphere, then \eqref{con2} easily yields that
\begin{center}
$||\nabla V||^2=\frac{(\epsilon-2)^2+\epsilon^2}{\epsilon}||V||^2-4\frac{\epsilon-1}{\epsilon} a^2.$
\end{center}
In the Riemannian case consider the local orthonormal
basis $\lbrace e_1,e_2,e_3\rbrace$ of vector fields, where $e_i$ is the base described in \eqref{base3}. After a similar and straightforward calculation we find that
\begin{center}
$||\nabla V||^2=\frac{(\epsilon-2)^2+\epsilon^2}{\epsilon}||V||^2+4\frac{\epsilon-1}{\epsilon} a^2.$
\end{center}
\end{proof}
We already know from Theorems \ref{hor4} and \ref{hor55} which vector fields in $\su(2)$ of Berger spheres
are critical points for
the energy functional. Taking into account Proposition \eqref{E1}, we then have the following. 
\begin{theorem}
Let $(\s^3, g_{\epsilon})$ be the Berger sphere (in both Riemannian and Lorntzian cases).
\begin{itemize}
\item[$(a)$] 
 $(2+(\frac{(\epsilon^2-2)^2+\epsilon^2}{2\epsilon})\rho^2)vol(\s^3, g_{\epsilon})$ is the minimum value of the energy functional $E$ restricted to vector fields of constant length $\rho$. Such a minimum is attained by all vector fields $V=be_2+ce_3\in \su(2)$ of length $||V||=\rho=\sqrt{b^2+c^2}$.
 \item[$(b)$] 
 $(2+\epsilon \rho^2)vol(\s^3, g_{\epsilon})$ is the minimum value of the energy functional $E$ restricted to vector fields of constant length $\rho$. Such a minimum is attained by all vector fields $V=ae_1\in \su(2)$ of length $||V||=\rho=\sqrt{a^2}$.
 \end{itemize}
\end{theorem}
As we can see, this section completes the results in \cite{g22}. \\

\end{document}